\documentclass[12pt]{amsart}
\usepackage{a4wide}
\usepackage{amscd}

\begin{document}
\numberwithin{equation}{section}
\newtheorem{cor}{Corollary}[section]
\newtheorem{theorem}[cor]{Theorem}
\newtheorem{prop}[cor]{Proposition}
\newtheorem{lemma}[cor]{Lemma}
\theoremstyle{definition}
\newtheorem{defi}[cor]{Definition}
\newtheorem{remark}[cor]{Remark}
\newtheorem{example}[cor]{Example}
\newcommand{\cC}{\mathcal{C}}
\newcommand{\cL}{\mathcal{L}}
\newcommand{\cH}{\mathcal{H}}
\newcommand{\cun}{\cC^{\infty}}
\newcommand{\cz}{{\mathbb C}}
\newcommand{\rz}{{\mathbb R}}
\newcommand{\fm}{f^{-1}}
\renewcommand{\index}{\mathrm{index}}
\newcommand{\ov}{\widetilde}
\newcommand{\oD}{\overline{D}}
\newcommand{\px}{\partial_x}
\newcommand{\RM}{{\mathbb R}}
\newcommand{\ZM}{{\mathbb Z}}
\newcommand{\CM}{{\mathbb C}}
\newcommand{\Spec}{\operatorname{Spec}}
\newcommand{\supp}{\mathrm{supp}}
\newcommand{\tf}{{\tilde{f}}}\newcommand{\tg}{{\tilde{g}}}
\newcommand{\tih}{\tilde{h}}
\newcommand{\tphi}{\tilde{\phi}}
\newcommand{\vol}{\mathrm{vol}}
\newcommand{\Xe}{X_\epsilon}

\def\f{\varphi}
\def\e{\varepsilon}
\def\d{\delta}
\def\O{\Omega}
\def\P{\mathrm{P}}
\def\a{\alpha}
\def\b{\beta}
\def\g{\gamma}
\def\la{\langle}
\def\ra{\rangle}

\def\SO{{\rm SO}}
\def\Spin{{\rm Spin}}
\def\U{{\rm U}}
\def\pso{P_{\SO(2m+1)}N}
\def\psp{P_{\Spin(2m+1)}N}
\def\psu{P_{\U(1)}N}
\def\psou{P_{\SO(2m+2)}N}
\def\psoub{P_{\SO(2m+2)}M}
\def\psub{P_{\U(1)}M}
\def\pspc{P_{\Spin^c(2m+1)}N}
\def\pspcu{P_{\Spin^c(2m+2)}N}
\def\pspcub{P_{\Spin^c(2m+2)}M}
\def\be{\begin{equation}}
\def\ee{\end{equation}}
\def\beq{\begin{eqnarray*}}
\def\eeq{\end{eqnarray*}}
\def\p{\Psi}
\def\.{{\cdot}}
\def\X{{X^*}}
\def\Y{{Y^*}}
\def\Z{{Z^*}}
\def\n{\nabla}
\def\dt{{\partial_t}}

\title{The Dirac operator on generalized Taub-NUT spaces}
\author{Andrei Moroianu}
\address{Centre de Math\'ematiques\\
\'Ecole Polytechnique\\
91128 Palaiseau Cedex\\
France}
\email{am@math.polytechnique.fr}
\author{Sergiu Moroianu}
\address{Institutul de Matematic\u{a} al Academiei Rom\^{a}ne\\
P.O. Box 1-764\\RO-014700
Bu\-cha\-rest, Romania}
\address{\c Scoala Normal\u a Superioar\u a 
Bucharest, calea Grivi\c tei 21, Bucharest, Romania}
\email{moroianu@alum.mit.edu}
\date{\today}
\begin{abstract}
We find sufficient conditions for the absence of harmonic $L^2$ spinors
on spin manifolds constructed as cone bundles over a compact K\"ahler base.
These conditions are fulfilled for certain perturbations of the 
Euclidean metric, and also for the generalized Taub-NUT metrics 
of Iwai-Katayama, thus proving a conjecture of Vi\c sinescu and the
second author. 

\end{abstract}
\subjclass[2000]{58J50, 58J20}
\keywords{Dirac operator, non-Fredholm $L^2$ index, generalized
  Taub-NUT metrics.} 
\maketitle

\section{Introduction}

The Taub-NUT metrics on $\rz^4$ and their generalizations by
Iwai-Katayama \cite{IK} provide a fruitful  
framework for the study of classical and quantum anomalies in the presence of
conserved quantities, see e.g.\ \cite{CV}. To describe these metrics, consider the sphere
$S^3$ as the unit sphere inside the quaternions. There exist then three 
orthogonal unit vector fields $I,J,K$ given by left translation with 
the unit quaternions $i,j,k$. The Berger metrics $g_\lambda$
on $S^3$ are defined by setting the length of $I,J$ to be $1$, and
that of $K$ to be $\lambda$. The Iwai-Katayama metrics on
$\RM^4\setminus\{0\}\simeq\RM^+\times S^3$ have the form
\be\label{tn}
g_{IK}=\gamma^2(t)(dt^2+4t^2 g_{\lambda(t)})
\ee
where 
\begin{align*}
\gamma(t)=\sqrt{\tfrac{a+bt}{t}},&&\lambda(t)=\tfrac{1}{\sqrt{1+ct+dt^2}}
\end{align*}
for positive constants $a,b,c,d$. The apparent singularity at the origin is 
removable.

We are interested here in the axial quantum anomaly already studied in 
\cite{CMV, MV}. It was found in \cite{CMV} that the axial anomaly, 
i.e., the 
difference between the number of null states of positive and of negative 
chirality on a ball or annular domain,
may become non-zero for suitable choices of the parameters of the metric 
and of the domain when we impose the Atiyah-Patodi-Singer spectral condition 
at the boundary. Remarkably, when the radius of the ball is sufficiently 
large the index was always $0$. It was further proved in \cite{MV} that on 
the whole space, although the Dirac operator is not Fredholm, it only has 
a finite number of null states. The method of proving the finiteness of 
the index in \cite{MV} relied on a general index formula due to Vaillant 
\cite{boris}, and on a comparison between harmonic spinors for a pair of 
conformal metrics. On the standard Taub-NUT space, which is hyperk\"ahler 
and therefore scalar-flat, it is easy to see that there are no harmonic 
$L^2$ spinors using the Lichnerowicz identity and the infiniteness 
of the volume. It was somehow 
natural to conjecture in \cite{MV} that the $L^2$ index of the Dirac 
operator corresponding to the generalized Taub-NUT metric is also zero. 
The motivation of the present work is to prove the above conjecture:

\begin{theorem}\label{th1}
There do not exist $L^2$ harmonic spinors on $\RM^4$ for the 
generalized Taub-NUT metrics. In particular, the $L^2$ index of the
Dirac operator 
vanishes.
\end{theorem}

As we just mentioned, for the standard Taub-NUT metric this has been proved in 
\cite{MV}. Our approach here is less analytic, and more geometric, than the 
previous attempt
described above. We exploit the rich symmetries of the metric to decompose the 
spinors in terms of frequencies along the fibers as in e.g.\ \cite{nistor},
and then further in terms of eigenvalues of an associated $\mathrm{spin}^c$ 
Dirac operator on $S^2$. We obtain a system of ordinary differential equations
which we show does not admit any $L^2$ solutions. There are similarities
with \cite{lott}, \cite{ANSER07} in the analysis of this system, but 
the essential difference is that large time behavior is not enough 
to rule out $L^2$ harmonic spinors and we must use also the behavior near 
the origin. The method is more general and we can prove
our results for a wider class of manifolds, constructed from a circle
fibration over a compact Hodge base. Although the one-point completion 
of such a manifold will not be in general a topological manifold, we
consider it as a singular complete metric space, the appropriate condition on 
spinors being boundedness in the $L^\infty$ norm near the singular point.
Our main result (Theorem \ref{main})
applies both to the Iwai-Katayama metrics and to Euclidean metrics, 
and to certain perturbations thereof.

The paper is organized as follows: In Section \ref{cfwp} we introduce the 
class of metrics studied in the rest of the paper. In Section \ref{sect3}
we relate geometric objects -- like the Levi-Civita connection and
the Dirac operator -- of a circle-fibered space to the
corresponding objects on the base, and we  
introduce the announced splitting into frequencies along the fibers. 
Section \ref{sect4} contains similar computations in the case of
warped products, 
introducing an extra variable corresponding to the radial direction,
and computing the corresponding spin$^c$ Dirac operator. The main
analytic result 
is stated and proved in Section \ref{sect5} by reducing the problem to 
a linear system of ordinary differential equations on the positive
real half-line, and a careful analysis
both near infinity and $0$ to exclude $L^2$ solutions.
Finally in Section \ref{sectap} we extend the result in a rather formal way 
to include the Iwai-Katayama generalized Taub-NUT metrics.

\subsection*{Acknowledgments} The authors are indebted to Mihai 
Vi\c sinescu for his help concerning generalized Taub-NUT metrics and
to the anonymous referee for valuable comments. The authors
acknowledge the support of the 
  Associated European Laboratory ``MathMode". A.M.\ was partially
  supported by ANR-10-BLAN 0105. 
S.M.\ was partially supported by grant PN-II-ID-PCE 1187/2009.

\section{Circle fibered warped products}\label{cfwp}

Let $(B,g_B,\O)$ be a compact K\"ahler manifold of real dimension $2m$. Let
$h$ denote the warped product metric $dt^2+\a ^2(t) g_B$ on
$N:=\RM^+\times B$ and let $p$ denote the projection $N\to B$.

Assume that $(B,g_B,\O)$ is a Hodge manifold, i.e., $[\O]\in 2\pi H^2(B,\ZM)$. 
The classical isomorphism of \v Cech cohomology groups
$ H^1(B,S^1)\simeq H^2(B,\ZM)$ shows the existence of
a Hermitian line bundle $L_0\to B$ with first Chern class
$c_1(L_0)=-[\O]/2\pi$. Let $M_0$ denote the circle bundle of $L_0$.
The projection $q:M_0\to B$ can be
viewed as a principal $S^1$-bundle. By Chern-Weil 
theory (cf. \cite{m}, Ch.\ 16 for instance) there exists an
imaginary-valued connection 1-form $i\xi$ on $M_0$ such that $d\xi=q^*\O$.

We define $L:=\RM^+\times L_0$ and $\pi:=\mathrm{id}\times
p$ the projection of $L$ onto $N$. 
Then $\pi:L\to N$ is a Hermitian line 
bundle over $N$ whose circle bundle, denoted by $M$, is just
$M:=\RM^+\times M_0$.  

We endow $M$ with the Riemannian metric
$g:=dt^2+\a ^2(t)(p\circ \pi)^*g_B+
\b^2(t)\xi\otimes\xi$ for some positive functions $\a$ and $\b$ 
defined on $\RM^+$. 

The Riemannian manifold $(M,g)$ obtained in this way will be referred
to as the {\em circle-fibered warped product} (CFWP) over the Hodge manifold 
$(B,g_B,\O)$, with warping functions $\a$ and $\b$. Notice that a CFWP
can be viewed 
either as a {\em generalized cylinder} of a family of metrics on the
$S^1$-bundle 
$M_0$ over $B$ (cf. Proposition \ref{pr1} below) or as a Riemannian submersion
with 1-dimensional fibres over a warped product
$\RM^+\times_{\a}B$. It is the latter  
point of view which will be useful in order to relate spinors on $M$ and $B$.

\begin{example}\label{ex1}
The flat space $\CM^{m+1}\setminus \{0\}$ is the CFWP over the 
complex projective space 
$(\CM {\rm P}^m, g_{FS},\O_{FS})$ endowed with
the Fubini-Study metric, with warping functions $\a(t)=\frac{t}{\sqrt{2}}$, 
$\b(t)=t$. The normalization $g_{FS}$ of the Fubini-Study metric used here is
the one with scalar curvature equal to $2m(m+1)$ or, equivalently, the
one for which the projection $\mathbb{S}^{2m+1}\to(\CM {\rm P}^m,
\frac12g_{FS})$ is a Riemannian submersion (cf.\ \cite{m}, Ch.\ 13).
\end{example}

\begin{example}\label{ex2}
The Taub-NUT metric on $\CM^2$ is conformal to the one-point completion 
of the CFWP over the standard 2-sphere 
of radius $1/\sqrt{2}$ with warping functions 
$\a(t)=\sqrt{2}t$,
$\b(t)=\frac{2t}{1+bt}$. More generally, the {\em generalized} Taub-NUT 
metrics of Iwai-Katayama on $\CM^2$ are conformal to the one-point 
completion of the CFWP over 
$(\CM{\rm P}^1,g_{FS})$, i.e., 
the standard $2$-sphere of radius $1/\sqrt{2}$,
with warping functions $\a(t)=\sqrt{2}t$, 
$\b(t)=\frac{2t}{\sqrt{1+ct+dt^2}}$ 
for some positive constants $c$ and $d$ (cf. \cite{MV}, p.\ 6576):
\be\label{mik}
g_{IK}=\tfrac{a+bt}{t}
\left(dt^2+\a^2(t)\pi^*g_{FS}+\b^2(t)\xi\otimes\xi\right).
\ee
\end{example}
By Remark \ref{rem4} below, these are actually examples of CFWP's.
\begin{prop}\label{pr1}
Let $(M,g)$ be the CFWP over a Hodge manifold $(B^{2m},g_B,\O)$ 
with warping functions $\a$ and $\b$ and assume that 
$\displaystyle\lim_{t\to 0}\a(t)=\lim_{t\to 0}\b(t)=0$. 
Let $d$ denote the distance on $M$ induced by $g$. 
Then the metric completion $(\hat M,d)$ of $(M,d)$ has exactly 
one extra point. If $g$ extends to a smooth metric on $\hat M$, 
then  $(B,g_B,\O)$ is the complex projective space endowed with 
the Fubini-Study metric, and  
\[\displaystyle\lim_{t\to 0}\tfrac{\a(t)}{t}-\tfrac{1}{\sqrt 2}
=\lim_{t\to 0}\tfrac{\b(t)}{t}-1=0.\]
\end{prop}
\begin{proof}
The Riemannian manifold $(M,g)$ will be viewed as a generalized
cylinder (cf.\ \cite{bgm05}) of the family of metrics $g_t:=\a
^2(t)q^*g_B+\b^2(t)\xi\otimes\xi$ on the $S^1$-bundle $M_0$ over $B$
(which is a compact manifold). The first statement follows immediately
from the fact that for every $x\in M_0$, the rays $\RM^+\times \{x\}$
are geodesics parametrized by arc-length on $M$. Assume now that $g$
extends smoothly to $\hat M$. The Gauss Lemma applied to a
neighborhood of the origin $t=0$ in $\hat M$ shows that the distance
spheres $(M_0, g_t)$ (renormalized by a factor $1/t^2$) tend to the
standard sphere $S^{2m+1}$ in the Gromov-Hausdorff topology. In other
words, there exist non-zero constants $\a_0$, $\b_0$ such that
$\lim_{t\to 0}\a(t)/t=\a_0$, $\displaystyle\lim_{t\to 0}\b(t)/t=\b_0$
and $\a_0 ^2q^*g_B+\b_0^2\xi\otimes\xi$ is the standard metric on
$S^{2m+1}$. On the other hand, this metric is by definition a
Riemannian submersion over $(B, \a_0 ^2g_B)$ with totally geodesic
fibres of length $2\pi\b_0$. 
Since the length of every closed geodesic on $S^{2m+1}$ is $2\pi$ we
get $\b_0=1$. The manifold $(B,\a_0 ^2g_B)$ is then the quotient of
the sphere by an isometric $S^1$ action, so $B=\CM \P^m$ and $\a_0
^2g_B=\tfrac12g_{FS}$ (cf. \cite{m}, Ch.13). The constant $\a_0$ is
determined by the normalization condition
$c_1(L_0)=-[\O_B]/2\pi$. Indeed, since $L_0=\RM^{2m+2}\setminus \{0\}$
is clearly the tautological bundle $\O(-1)$ of $\CM \P^m$, its first
Chern class is equal to $-[\O_{FS}]/2\pi$, whence $\a_0=1/\sqrt2$. 
\end{proof}

The converse holds under some extra smoothness assumption on $\a$ and $\b$ at $t=0$ but we will not need this in the sequel.

\begin{remark}\label{rem4}
A metric conformal to a CFWP is itself a CFWP provided that the conformal 
factor is a radial function (i.e., it only depends on $t$). Indeed, if 
\[g=\g(t)^2(dt^2+\a ^2(t)\pi^*g_B+\b^2(t)\xi\otimes\xi),\]
in the new coordinate $s:=s(t)$ defined by 
$s=\int_0^t\g(u)du$, $g$ reads
\[g=ds^2+\a ^2(t(s))\pi^*g_B+\b^2(t(s))\xi\otimes\xi.\] 
The generalized Taub-NUT metrics from Example \ref{ex2} are thus particular 
cases of CFWP. We will analyze these metrics in more detail in Section
\ref{sectap}.
\end{remark} 

Our main goal in this paper will be to study the $L^2$-index 
of the Dirac operator on a CFWP $(M,g)$ when $M$ is a spin manifold. 
As we will see below, this is automatically the case when $B$ has 
a spin$^c$ structure whose auxiliary bundle is some tensor power of 
$L_0$, i.e., if the second Stiefel-Whitney class of $B$ satisfies
$w_2(B)=0$ or $w_2(B)\equiv c_1(L_0)$ mod $2$. In the next two 
sections we will relate spinors on $M$ to spin$^c$ spinors on $N$ and 
then further to spin$^c$ spinors on $B$. The results are quite general 
and can be viewed as a natural extension of the theory of projectable spinors 
introduced in \cite{m95} to the case of submersions with non-totally 
geodesic fibres.

\section{Spinors on circle fibrations}\label{sect3}

Let $\pi:(M,g)\to (N,h)$ be a Riemannian submersion with 1-dimensional fibres
of length $2\pi\b$ for some function $\b:N\to\RM^+$. 
The fibers of $\pi$ are totally geodesic if and only if $\b$ is constant, 
but we will mostly be interested in examples with non-constant $\b$ 
in the sequel. 

We can view $M$ as a principal $S^1$-fibration over $N$. Indeed, the flow 
$\f_t$ of the vertical Killing vector field $V$ on $M$ of length $\pi^*\b$ 
closes up at time $t=2\pi$, i.e., $\f_{2\pi}=id_M$, thus it defines a free 
$S^1$-action on $M$ whose orbit space is $N$. We denote by
$\psu$ this principal $S^1$-bundle with total space $M$.
The Riemannian metric $g$ can be written as $g=\pi^*h+\b^2(t)\xi\otimes\xi$, 
where $\xi$ is the 1-form on $M$ defined by $\xi(V)=1$ and $\ker \xi=V^\perp$.

The 2-form $d\xi$ is basic, i.e., there exists some 2-form $F$ on $N$ such that
$d\xi=\pi^*F$. This follows immediately from the Cartan formula 
and the fact that $V$ is Killing, or alternately since $i\xi$ is 
a connection $1$-form in the principal bundle $\psu$ (cf. Section 2).

The following result holds without restriction on the dimension of $N$ 
but we will state it only for the case we will need in the sequel.

\begin{lemma}\label{pb} Let $\psu\to N$ be the principal $S^1$-bundle over 
the $2m+1$-dimensional manifold $N$
defined by the Riemannian submersion $\pi:M\to N$. Let $L\to N$ be the
complex line bundle associated to $\psu$ with respect to the canonical
representation of $S^1$ on $\CM$. Then every
spin$^c$ structure $\pspc$ on $N$ with auxiliary bundle $L^{\otimes k}$, 
$k\in \ZM$ induces
a spin structure on $M$ and all these spin structures are isomorphic.
\end{lemma}

\begin{proof} 
By enlargement of the structure groups, the two-fold covering
$$\theta:\pspc\to\pso\times\psu$$
gives a two-fold covering
$$\theta:\pspcu\to\psou\times\psu,$$
which, by pull-back through $\pi$, gives rise to a $\hbox{Spin}^c$
structure on $M$:

\[
\begin{CD}
\pspcub 		@>{\pi}>> 	& \pspcu \\
@V{\pi^*\theta}VV 			& @V{\theta}VV\\
\psoub\times\psub 	@>{\pi}>> 	& \psou\times\psu\\
@VVV					& @VVV\\
M			@>{\pi}>>	& P.
\end{CD}
\]  

This construction actually yields a {\it spin} structure on
$M$. Indeed, the pull back $\psub$ to $M$ of $\psu$ is trivial since it
carries a tautological global section $\sigma(u)=(u,u),\ \forall u\in M=\psu$.
Correspondingly, the pull-back to $M$ of every associated bundle $L^{\otimes k}$ 
is trivial.
\end{proof}

From now on we assume that $N$ carries some spin$^c$ structure with 
auxiliary bundle $L^{\otimes k}$, and we study $M$ with the spin structure 
induced by the previous lemma. In particular, we will consider the 
flat connection on the trivial bundle $\psub$, rather than the pull-back 
connection from $\psu$, in order to define covariant derivatives of spinors 
on $M$. The following result, first proved in \cite{m98}, relates an
arbitrary connection on a principal bundle $\pi:M=P_{\U(1)}N\to N$ and the flat
connection on $\pi^*M=P_{\U(1)}M\to M$.
\[
\begin{CD}
\pi^*M=P_{\U(1)}M\simeq M\times S^1 @>{ \pi}>>
& M=P_{\U(1)}N\\
@V{\pi^*\pi}VV&@V{\pi}VV\\
M@>{\pi}>> & N
\end{CD}
\]  

\begin{lemma} \label{hor} The connection form $A_0$ of the flat
  connection on $P_{\U(1)}M$ can be related to an arbitrary connection
  $A$ on $P_{\U(1)}N$ by 
$$ A_0((\pi^*s)_*(U))=-A(U),$$
$$ A_0((\pi^*s)_*(X^*))=A(s_*X),$$
where $U$ is a vertical vector field on $M$, $X^*$ is the horizontal
lift (with respect to $A$) of a vector field $X$ on $N$, and $s$ is a
local section of $M\to N$. 
\end{lemma}
  
\begin{proof} The identification $M\times {\U(1)}\simeq \pi^*M$
is given by $(u,a)\mapsto (u,ua)$, for all $(u,a)\in M\times
{\U(1)}$. For some fixed $u\in M$, take a path $u_t$ in the
fiber over $x:=\pi(u)$ such that $u_0=u$ and $\dot{u}_0=U$. 
We define $a_t\in {\U(1)}$ by $u_t=s(x)a_t$, so via the above
identification we have  
$$(\pi^*s)(u_t)=(u_t,s(x))=(u_t,(a_t)^{-1}),$$
and thus 
$$A_0((\pi^*s)_*(U))=-a_0^{-1}\dot{a}_0=-A(\dot{u}_0)=-A(U).$$

Similarly, for $x\in N$ and $X\in T_x N$, take a path $x_t$ in $N$
such that $x_0=x$ and $\dot{x}_0=X$. Let $u\in\pi^{-1}(x)$ and $u_t$
the horizontal lift of $x_t$ such that $u_0=u$. We define $a_t\in
{\U(1)}$ by $s(x_t)=u_ta_t$, which by derivation gives
$s_*(X)=R_{a_0}\dot{u}_0+u_0\dot{a}_0$. Then 
$$(\pi^*s)(u_t)=(u_t,s(x_t))=(u_t,a_t),$$
and thus, using the fact that $\dot{u}_0$ is horizontal,
\[A_0((\pi^*s)_*(X^*))=a_0^{-1}\dot{a}_0=A(s_*(X)).\qedhere\]
\end{proof}

Recall that the complex Clifford representation $\Sigma_{2m+2}=
\Sigma ^+_{2m+2}\oplus\Sigma ^-_{2m+2}$ can be identified with
$\Sigma_{2m+1}\oplus\Sigma_{2m+1}$ by defining in an orthonormal basis
\[e_j\cdot(\psi,\phi)=
\begin{cases}
	(e_j\cdot\phi, e_j\cdot\psi) &  \text{for $j\le 2m+1$}\\
	(-\phi,\psi) & \text{for $j=2m+2$}.	
\end{cases}\]     
 
Accordingly, we obtain identifications, denoted by $\pi ^\pm$,
of the pull back 
$\pi^*\Sigma N$ with $\Sigma ^\pm M$. By a slight abuse of notation we will 
denote $\pi^\pm$ and $\Sigma ^\pm M$ by $\pi^\e$ and $\Sigma ^\e M$ for $\e=\pm1$. With respect
to these identifications, 
if $X$ is a vector  and $\p$ is a spinor on $N$, then  
\begin{align} 
\label{ps1}\X\.\pi^\e\p=&\ \pi^{-\e}(X\.\p),\\
\label{ps2}\tfrac{1}{\b}V\.(\pi^\e\p)=&\ \e \pi^{-\e} \p, 
\end{align}
where $\frac{1}{\b}V$ is the unit vertical vector field defined at the beginning
of this section, and $\X$ denotes the horizontal lift 
to $M$ of a vector field $X$ on $N$.

We consider now a spin$^c$ structure $\pspc$ on $(N,h)$ with
auxiliary bundle 
$L^{\otimes k}$ and denote by $\n^N$ the covariant derivative induced on
$\Sigma N$ by the connection form $i\xi$ of $\psu$. By Lemma \ref{pb},
the pull-back to $M$ of $\pspc$ 
induces by enlargement a spin structure on $(M,g)$, where we recall that
$g=\pi^*h+\b^2\xi\otimes\xi$.   

\begin{prop} Let $\n^M$ denote the covariant derivative on $\Sigma ^\e M$ induced by
  the Levi-Civita connection on $(M,g)$ and the flat connection
  on $\pi^*\psu$. Let $\n^N$ denote the spin$^c$ covariant derivative
  on $\Sigma N$ induced by
  the Levi-Civita connection on $(N,h)$ and the connection
  form $A=i\xi$ on $\psu$. Then $\n^M$ and $\n^N$ are related by 
\begin{align}\label{3}
\n^M_\X(\pi^\e\p)=&\ 
\pi^\e(\n^N_X\p-\tfrac{\e\b}{4}T(X)\.\p),&&
\forall X\in TM,\\
\label{4} \n^M_{V}(\pi^\e\p)=&\ \pi^\e\left(\tfrac{\b^2}4
F\.\p+\tfrac{\e}{2}d\b\.\p-\tfrac{ki}{2}\p\right),
\end{align}
where $T$ is the endomorphism of $TN$ defined by $d\xi(\X,\Y)=F(X,Y)=h(TX,Y)$.
\end{prop}

\begin{proof} 
If $V$ denotes as before the vertical vector field such that
$\xi(V)=1$, the Koszul formula and the fact that
$[V,\X]=0$ for all vector fields $X$ on $N$ yield 
\begin{align}
\label{a1}	g(\n^M_\X\Y,\Z)=&\ h(\n^N_XY,Z)\\
\label{a2}	\begin{split}g(\n^M_V\X ,\Y) =&\ g(\n^M_\X
  V,\Y)=-\tfrac12g(V,[\X,\Y])=-\tfrac{\b^2}2\xi([\X,\Y]) \\
=&\ \tfrac{\b^2}2d\xi(\X,\Y)=\tfrac{\b^2}2h(TX,Y)=\tfrac{\b^2}2 F(X,Y),
\end{split}\\
\intertext{and} 
\label{a3}g(\n^M_V\X ,V)=&\ g(\n^M_\X V,V)=\b X(\b),
\end{align}
for all vector fields $X$, $Y$ and $Z$ on $N$.

Consider a spinor field on $N$ locally expressed as $\p=[\sigma,\psi]$,
where $\psi:U\subset N\to\Sigma_{2m+1}$ is a vector-valued function,
and $\sigma$ is a local section of $\pspc$ whose projection onto
$\pso$ is a local orthonormal frame $(X_1,...,X_{2m+1})$ and whose
projection onto $\psu$ is a local section $s$. Then $\pi^\e\p$ can be
expressed as $\pi^\e\p=[\pi^*\sigma,\pi^*\xi]$. Moreover, 
the projection of $\pi^*\sigma$ onto $\psoub$ is the local
orthonormal frame $(\frac{1}{\b}V,X_1^*,...,X_{2m+1}^*)$ and its
projection onto $\psub$ is just $\pi^*s$.  

Using the general formula for the covariant derivative on spinors,
Lemma \ref{hor}, and the
fact that the bundle $L^{\otimes k}$ is associated to $\psu$ via the
  representation $\rho^k(z)=z^k$ of $S^1$ on $\CM$,
 we obtain 
\begin{align*}
\n^M_{\X}\pi^\e\p=&\ [\pi^*\sigma,\X(\pi^*\psi)]+
\tfrac{1}{2}\sum_{j<k}g(\n^M_{\X}X_j^*,  
X_k^*)X_j^*\.X_k^*\.\pi^\e\p\\
&+\tfrac{1}{2}\sum_j g(\n^M_\X
X_j^*,\tfrac{1}{\b}V)X_j^*\.\tfrac{1}{\b}V\.\pi^\e\p 
+\tfrac{1}{2}\rho^k_*(A_0((\pi^*s)_*\X))\pi^\e\p\\   
=&\ [\pi^*\sigma,\pi^*(X(\psi))]+
\tfrac{1}{2}\sum_{j<k}h(\n^N_XX_j,X_k)\pi^\e(X_j\.X_k\.\p)+
\tfrac{1}{2}\rho^k_*(A(s_*X))\pi^\e\p\\
&-\tfrac{\e\b}{4}\sum_j h(T(X),X_j)\pi^\e(X_j\.\p)=\ \pi^\e\left(\n^N_X\p-\tfrac{\e\b}{4}T(X)\.\p\right). 
\end{align*}
and similarly, since $A=i\xi$, 
\begin{align*}
\n^M_{V}(\pi^\e\p)=&\ [\pi^*\sigma,V(\pi^*\psi)]+
\tfrac12\sum_{j<k}g(\n^M_{V}X_j^*, X_k^*)X_j^*\.X_k^*\.\pi^\e\p\\
&+\tfrac12\sum_j g(\n^M_VX_j^*,\tfrac1{\b}V)X_j^*\.\tfrac1{\b}V\.\pi^\e\p+
\tfrac12\rho^k_*(A_0((\pi^*s)_*V))\pi^\e\p\\  
=&\ \tfrac{\b^2}4\sum_{j<k}F(X_j,X_k)
\pi^\e(X_j\.X_k\.\p)+\e\tfrac12\sum_jX_j(\b)\pi 
^{\e}(X_j\.\Psi)
-\tfrac12\rho^k_*(A(V))\pi^\e\p\\
=&\ \tfrac{\b^2}4\pi^\e(F\.\p)+\tfrac{\e}2\pi
^{\e}(d\b\.\Psi)-\tfrac{ki}{2}\pi^\e\p=
\pi^\e\left(\tfrac{\b^2}4 F\.\p+\tfrac{\e}{2}d\b\.\Psi-\tfrac{ki}{2}\p\right).
\qedhere
\end{align*}
\end{proof}

\begin{cor} The Dirac operators on $M$ and $N$ are related by
\be \label{dirac} D^M(\pi ^\e\Psi)=\pi^{-\e}\left(D^N\p-
\tfrac{\e\b}{4}F\.\p+\tfrac{d\b}{2\b}\.\p-\tfrac{\e ki}{2\b}\p\right)  
\ee
\end{cor}
\begin{proof} Simple computation using (\ref{ps1})--(\ref{4}):
\begin{align*} 
D^M(\pi ^\e\Psi)=&\sum_j X_j^*\.\n^M_{X_j^*}(\pi ^\e\Psi)
+\tfrac{1}{\b}V\.\n^M_{\frac{1}{\b}V}(\pi^\e\p)\\
=&\sum_j X_j^*\.\pi^\e\left(\n^N_{X_j}\p-
\tfrac{\e\b}{4}T(X_j)\.\p\right)
+\tfrac{1}{\b}V\.\tfrac{1}{\b}\pi^\e\left(\tfrac{\b^2}4
F\.\p+\tfrac{\e}{2}d\b\.\Psi-\tfrac{ki}{2}\p\right)\\
=&\ \pi^{-\e}\left(D^N\p-\tfrac{\e\b}{4}X_j\.T(X_j)\.\p+\tfrac{\e\b}{4}
F\.\p+\tfrac{d\b}{2\b}\.\p-\tfrac{\e ki}{2\b}\p\right)\\
=&\ \pi^{-\e}\left(D^N\p-
\tfrac{\e\b}{4}F\.\p+\tfrac{d\b}{2\b}\.\p-\tfrac{\e ki}{2\b}\p\right).\qedhere
\end{align*}
\end{proof}

Note that Equation \eqref{4} is just the Bourguignon-Gauduchon formula 
\cite{bg} for the Lie derivative of a spinor field with respect to 
a Killing vector field:
$$\n_V\Phi=\cL_V\Phi+\tfrac14 dV^\flat\.\Phi.$$
Incidentally, this formula shows that if $\Phi=\pi^\e\p$ is the pull-back 
of a spin$^c$ spinor on $N$ corresponding to a spin$^c$ structure with 
auxiliary bundle $L^{\otimes k}$, then $\cL_V\Phi=-\frac{ik}2\Phi$. 
Yet another way to understand this fact is the following. A section $s$ 
of the complex line bundle $L\to N$ can be identified with a complex function 
$f_s$ on $M\subset L$ with ``frequency" $-1$ by $s(\pi(x))=xf_s(x)$ 
(clearly $f_s(\f_t(x))=e^{-it}f_s(x)$ so $\cL_V(f_s)=-if_s$). 
Now, a spin$^c$ bundle on $N$ with auxiliary bundle $L^{\otimes k}$ 
is the tensor product between two locally defined bundles: the spin bundle 
of $N$ and a square root of $L^{\otimes k}$. It is now clear 
that the pull-back to $M$ of its sections are spinors on $M$ with 
``frequency" $-k/2$.

\begin{lemma}\label{desc}
Let $\cH$ be the Hilbert space of $L^2$ spinors on $M$ and let $\cH_n$
be the Hilbert space of $L^2$ sections of the  spin$^c$ bundle on $N$
with auxiliary bundle $L^{\otimes n}$.  
If the spin structure of $M$ is induced as before by a spin$^c$
structure on $N$ with auxiliary bundle $L^{\otimes k}$, $\cH$
decomposes in a Hilbertian 
orthogonal direct sum 
$$\cH=\bigoplus_{n\in \ZM,\e=\pm1} \pi^\e(\cH_{k+2n}).$$
If $\Phi\in\cH$ is smooth, the same holds for its components
$\Phi_n\in\cH_{k+2n}$ and the length of $\Phi_n$ at any $x\in M$ is
bounded by the maximum of the lengths of $\Phi$ along the $S^1$-orbit
of $x$. 
\end{lemma}

\begin{proof}
The space of $L^2$ functions on $M$ decomposes in a Hilbertian
orthogonal direct sum  
$$L^2(M)=\bigoplus_{n\in \ZM} L^2_n(M),$$
where $L^2_n(M):=\{f\in L^2(M)\ | \ \cL_V(f)=in\cdot f\}$ is the space
of $L^2$ functions on $M$ of frequency $n$, identified as before with
the space of $L^2$ sections of $(L^*)^{\otimes n}$. The desired
decomposition now follows immediately from the fact that the tensor
product between $L^{\otimes n}$ and the spin$^c$ bundle on $N$ with
auxiliary bundle $L^{\otimes k}$ is the spin$^c$ bundle on $N$ with
auxiliary bundle $L^{\otimes (k+2n)}$.  

The last assertion follows from the fact that 
\[\Phi_n(x)=\frac1{2\pi}\int_0^{2\pi}e^{-int}((\f_t)_*\Phi)(x) dt.\qedhere\]
\end{proof}

A similar decomposition of the space of spinors on total space of $S^1$ 
fibrations was used by Ammann \cite{a} and Ammann and B\"ar \cite{ab}
in order to study the properties  
of the spectrum of the Dirac operator when the fibres collapse, 
and also by Nistor \cite{nistor} in his study of the $S^1$-equivariant index.

Note that Eqs.\ \eqref{3} and \eqref{4} already appeared (in a
slightly different form because of different spinor identifications)
as Lemma 3.2 and Eq.\ (2) respectively in Ammann \cite{a}. However,
since the proofs of these formulas appear only in Ammann's thesis
\cite{tha}, we have chosen for the reader's convenience to include
here the full details of the proofs.

\section{Spinors on warped products} \label{sect4}

Let now $(B,g_B)$ be a Riemannian manifold and assume that $(N,h)$ is
the warped product $(\RM^+\times 
B,dt^2+\a(t)^2g_B)$ for some positive function $\a:\RM^+\to \RM^+$. We
denote by $p:N\to B$ the standard 
projection. Let $\n^N$ and $\n^B$ denote the Levi-Civita covariant
derivatives on 
$N$ and $B$ and let $\dt$ denote the (unit) radial vector field on
$N$. Every vector field $X$ on $B$ defines a ``horizontal" vector
field also denoted by $X$ on $N$ such that $[X,\dt]=0$.

The warped product formulae for the covariant derivatives
(\cite{ON83}, p.206) are
\begin{align} \label{pt1} \n^N_{\dt}\dt=&\ 0,\\
\label{pt2}\n^N_{\dt}X=&\ \n^N_X\dt=\tfrac{\a'}{\a}X,\\
\label{pt3} \n^N_XY=&\ \n^B_XY-\a\a'g_B(X,Y)\dt.
\end{align} 

Consider a spin$^c$ structure on $B$ with auxiliary line bundle
$L_0^{\otimes k}$ and 
the induced pull-back spin$^c$ structure on $N$
with auxiliary line bundle $L^{\otimes k}$, where $L=p^*L_0$ is the
pull-back of $L_0$. We 
continue to denote by $\n^B$ and $\n^N$ the spin$^c$ covariant
derivatives induced by some connection on $L_0$.

Assume now that $B$ has even dimension.
The spinor bundle $\Sigma N$ can be canonically identified with
$\pi^*(\Sigma B)$ such  
that the Clifford product satisfies  
\begin{align}\label{cl1}
\tfrac{1}{\a}X\.(p^*\p)=&\ p^*(X\.\p),&&\forall  X\in TB,\\
\intertext{and}
\label{cl2} \dt\.(p^*\p)=&\ ip^*(\bar\p),\\
\intertext{where $\bar\p:=\p_+-\p_-$ is the ``conjugate" of
  $\p=\p_++\p_-$ with respect to the chiral decomposition $\Sigma
  M=\Sigma ^+M\oplus \Sigma ^-M$.
   From \eqref{cl1} we easily obtain}
\label{cl3}\a^q(p ^*\omega)\.(p^*\p)=&\ p^*(\omega\.\p)
\end{align}
for every $q$-form $\omega$ on $B$. 
Using the warped product formulae
one can easily relate the spin$^c$ covariant derivatives $\n^N$ and
$\n^B$ like before:
\begin{align}\n^N_X(p^*\p)=&\ p^*(\n^B_{X}\p-
  \tfrac{1}{2}i\a'X\.\bar\p),&& \forall X\in TB,\\
\intertext{and}
\n^N_{\dt}(p^*\p)=&\ 0.\\
\intertext{In particular, the Dirac operators on $N$ and $B$ are related by}
\label{dirac1}
  D^N(p^*\p)=&\ \tfrac{1}{\a}p^*(D^B\p+im\a'\bar\p).
\end{align}
Indeed, if $(X_1,\ldots,X_{2m})$ is a local orthonormal basis on $B$,
then $(\frac{1}{\a}X_1,\ldots,\frac{1}{\a}X_{2m},\dt)$ is a local
orthonormal basis on $N$, whence
\begin{align*}  
D^N(p^*\p)=&\ \sum_j \tfrac{1}{\a}X_j\.\n^N_{\frac{1}{\a}X_j}(p ^*\Psi)
+\dt\.\n^N_{\dt}(p^*\p)\\
=&\ \sum_j
\tfrac{1}{\a}X_j\.\left(\tfrac{1}{\a}p^*(\n^B_{X_j}\p-
\tfrac{1}{2}i\a'X_j\.\bar\p)\right)\\
=&\ \tfrac{1}{\a}p^*(D^B\p+im\a'\bar\p).
\end{align*}
More generally, one can identify a spinor $\Psi$ on $N$ with a
1-parameter family $\p_t$ of spinors on $B$, and  (\ref{dirac1})
becomes
\be\label{dirac2}
  D^N\p=p^*\left(\tfrac{1}{\a}D^B\p_t+
\tfrac{im\a'}{\a}\bar\p_t+i{\dot{\bar\p}_t}\right). 
\ee

\section{Harmonic spinors on CFWP's}\label{sect5}

We now have all necessary ingredients in order to prove the main result 
of this paper:

\begin{theorem}\label{main}
Let $(M,g)$, $g=dt^2+\a ^2(t)\pi^*g_B+\b^2(t)\xi\otimes\xi$ be a
circle-fibered warped product (CFWP) over the Hodge manifold
$(B^{2m},g_B,\O)$,  
endowed with the spin structure defined by some spin$^c$ structure on
$B$ as before. 
Assume that the positive warping functions $\a$ and $\b$ satisfy
the conditions
\begin{enumerate}
\item[(a)] $\displaystyle\int_x^\infty
    e^{-\int_x^t\frac{1}{\sqrt 2\a(s)}ds}dt=\infty$
for all $x>0$;
\item[(b)] $\displaystyle\lim_{t\to 0} \a(t)=0$;
\item[(c)] $\displaystyle2\a^2(t)\geq \b^2(t)> \tfrac{m-1}{m}2\a^2(t)$ for all $t>0$.
\end{enumerate}
Then $(M,g)$ carries no non-trivial harmonic $L^2$ spinors which are 
bounded near the singularity $t=0$.
\end{theorem}
\begin{proof}
Assume that $\Phi$ is a non-zero harmonic $L^2$ spinor, 
bounded near $t=0$. By elliptic regularity, $\Phi$ is smooth. 
We can of course assume that $\Phi$ is chiral, i.e., it is a section 
of $\Sigma^\e M$ for some $\e=\pm1$. By Lemma \ref{desc}, 
and using the fact that the Dirac operator commutes with 
the Lie derivative $\cL_V$, we can also assume that $\Phi=\pi^\e\p$ 
is the pull-back of some $L^2$ section $\p$ of the spin$^c$ structure 
on $N$ with auxiliary bundle $L^{\otimes k}$ (the integer $k$ is even or odd, 
depending on whether the spin structure on $M$ projects onto 
a spin structure on $N$ or not).

Using \eqref{dirac} we infer
\begin{equation}\label{har}
D^N\p=\tfrac{1}{4\b}(\e \b^2F\.\p-2d\b\.\p+\e 2ik\p).
\end{equation}
We now view $\p$ as a family $\p_t$ of spin$^c$ spinors on $B$.
Recalling that $F=p^*\O$ and taking (\ref{har}),
(\ref{cl3}) and (\ref{dirac2}) into account, we get
\be\label{eq1}\left(\tfrac{1}{\a}D^B\p_t+
\tfrac{im\a'}{\a}\bar\p_t+i{\dot{\bar\p}_t}\right)=
\tfrac{1}{4\b}\left(\e\tfrac{\b^2}{\a ^2}\Omega\.\p_t-2i\b'\bar\p_t+\e
  2ik\p_t\right) 
\ee
The spin$^c$ bundle $\Sigma B$ decomposes in a direct
sum (cf. \cite{ki})
$$\Sigma B=\bigoplus_{l=0}^m\Sigma ^lB$$
of eigenspaces of the operator of Clifford multiplication by the K\"ahler form
$\Omega$, i.e., 
$$\Sigma ^lB=\{\p\in\Sigma B\ |\ \Omega\.\p=i(2l-m)\p\}.$$
One has  $\Sigma ^lB\subset \Sigma ^+B$ if $l$ is even and 
$\Sigma ^lB\subset \Sigma ^-B$ if $l$ is odd. Moreover $D^B$ maps
sections of $\Sigma ^lB$ to sections of $\Sigma ^{l-1}B\oplus\Sigma
^{l+1}B$ and each $\Sigma ^lB$ is stable by $(D^B)^2$. This easily
shows that every eigenspinor of $D^B$ is a finite sum of eigenspinors
of $D^B$ in $C^\infty(\Sigma ^lB\oplus \Sigma 
^{l+1}B)$ for $0\le l\le m-1$.

Since $B$ is compact and $D^B$ is elliptic, the space of $L^2$ spinors
on $B$ is the Hilbertian direct sum of the eigenspaces of $D^B$. By
the above, there exits $l\in\{0,\ldots, m-1\}$, $\lambda\in\RM$ and  
a spinor $\Phi=\Phi_l+\Phi_{l+1}\in C^\infty(\Sigma ^lB\oplus \Sigma
^{l+1}B)$ with $D^B\Phi=\lambda\Phi$ such that the functions
\begin{align*}
u(t):=\int_B\la\p_t,\Phi_l\ra dv_B,&&
v(t):=\int_B\la\p_t,\Phi_{l+1}\ra dv_B
\end{align*}
do not vanish identically.

Taking the scalar product with $\Phi_l$ and $\Phi_{l+1}$ in
(\ref{eq1}) and integrating over $B$ yields
\begin{align*}
\tfrac{\lambda}{\a}v+(-1)^l\tfrac{im\a'}{\a}u+(-1)^liu'=&\ 
\e\tfrac{i(2l-m)\b^2}{4\b\a ^2}u
-(-1)^l\tfrac{i\b'}{2\b}u+\e\tfrac{ik}{2\b}u\\
\tfrac{\lambda}{\a}u-(-1)^l\tfrac{im\a'}{\a}v-(-1)^liv'=&\ 
\e\tfrac{i(2(l+1)-m)\b^2}{4\b\a^2}v
+(-1)^l\tfrac{i\b'}{2\b}v+\e\tfrac{ik}{2\b}v
\end{align*}
which can be written after setting $w:=iv$:
\begin{align*}
u'=&\left(\tfrac{\e(-1)^l(2l-m)\b^2-2\a ^2\b'+
\e(-1)^l2\a ^2 k-4m\a\a'\b}{4\b\a
  ^2}\right)u+(-1)^l\tfrac{\lambda}{\a}w\\ 
w'=&\ (-1)^l\tfrac{\lambda}{\a}u+
\left(\tfrac{\e(-1)^l(m-2(l+1))\b^2-2\a ^2\b'-\e 
(-1)^l2\a ^2 k-4m\a\a'\b}{4\b\a
  ^2}\right)w.
\end{align*}
Denoting $U:=u \b^{\frac12}\a ^m$ and $W:=w \b^{\frac12}\a ^m$ this
simplifies to
\be\label{sy}\begin{split}
(-1)^lU'=&\ \e \left(\tfrac{(2l-m)\b^2+2\a ^2 k}{4\b\a
  ^2}\right)U+\tfrac{\lambda}{\a}W\\
(-1)^lW'=&\ \tfrac{\lambda}{\a}U+
\e \left(\tfrac{(m-2(l+1))\b^2-2\a ^2 k}{4\b\a^2}\right)W.
\end{split}
\ee

We have shown that if $(M,g)$ carries a non-trivial harmonic spinor, 
then \eqref{sy} has a non-trivial solution $(U,W)$ for some
$l\in\{0,\ldots,m-1\}$, $k\in\ZM$, $\lambda\in \mathrm{Spec}(D^B)\subset\RM$, 
and $\e\in\{\pm 1\}$. Moreover, since the volume form on $(M,g)$ is
$$dv_M=\a^{2m}\b dt\wedge\xi\wedge dv_B$$
and $B$ is compact, Fubini's theorem shows that the original spinor
$\Phi$ is $L^2$ on $M$ if and only if  
\begin{align*}
\int_0^\infty |\a^m\b^{\frac12}\p_t(x)|^2dt<\infty,&&\forall x\in B.
\end{align*}
We thus get that
$U$ and $W$ are $L^2$ functions on $\RM^+$ and satisfy
$U,W\in O(\a^m\b^{\frac{1}{2}})$ at $t=0$.

From conditions (b), (c) and the definition of $U,V$ we get
\be\label{lim} \lim_{t\to 0}U(t)=\lim_{t\to 0}W(t)=0.
\ee
The system \eqref{sy} reads 
\be\label{sy1}\begin{cases}
U'=\rho U+\sigma W\\
W'=\sigma U +\tau W\\
\end{cases}
\ee
where 
\begin{align*}\rho:=\ \e(-1)^l\left(\tfrac{(2l-m)\b^2+2\a ^2 k}{4\b\a
  ^2}\right),&& \tau:=\ \e
(-1)^l\left(\tfrac{(m-2(l+1))\b^2-2\a ^2 k}{4\b\a
  ^2}\right),&& \sigma:=\ (-1)^l\tfrac{\lambda}{\a}.\end{align*}
  
Notice that the coefficients of the system \eqref{sy1} are real
functions, thus 
we can assume that $U,W$ are real by considering separately their 
real and imaginary parts.

\begin{lemma}\label{sig} 
If a linear combination of the functions $U$ and $W$ is monotonous, 
it must vanish identically.
\end{lemma} 
\begin{proof} Using \eqref{lim} we see that if $aU+bW$ does not 
vanish identically, then $|aU+bW|$ is bounded from below by a non-zero 
constant on $[x_0,\infty)$ for some $x_0>0$, so it cannot be $L^2$.
\end{proof}

The previous lemma 
together with \eqref{sy1} show that $\lambda\ne0$: indeed, for $\lambda=0$ 
the system \eqref{sy1} uncouples into two first-order linear ODE's, 
whose nontrivial solutions never vanish by uniqueness, hence 
they have constant sign and so Lemma \ref{sig} applies. 
By changing $U$ to $-U$ if necessary, we can therefore assume that 
$\sigma(t)> 0$ for all $t>0$. 
\begin{lemma}\label{pos} 
If $\sigma(t)>0$ for all $t>0$ then we must have
$(UW)(t)\le 0$ for all $t\in\RM^+$.
\end{lemma} 
\begin{proof}
Assume that $UW>0$ on some open interval $I$. From (c) we easily infer 
\be\label{ine}
\tau+\rho = -\varepsilon (-1)^l\tfrac{\beta}{2\alpha^2}\ge
-\tfrac{1}{\sqrt{2}\a}, \ee 
so \eqref{sy1} yields 
\be\label{in}(UW)'=(\tau+\rho)UW+\sigma(U^2+W^2)\ge -\tfrac{1}{\sqrt 2\a}UW.\ee
Consider the maximal interval $J:=(x_0,x_1)$ containing $I$ on which
$UW>0$. For every $x_0<x\leq t<x_1$, \eqref{in} implies
\begin{equation}\label{uwnega}
(UW)(t)\ge (UW)(x)e^{-\int_x^{t}\frac{1}{\sqrt 2\a(s)}ds}.
\end{equation}
If $x_1<\infty$ then by continuity 
$(UW)(x_1)\ge(UW)(x)e^{-\int_x^{x_1} \frac{1}{\sqrt 2\a(s)}ds}>0$,
contradicting the maximality of $J$. Therefore $x_1=\infty$, so $UW(t)>0$
for all $t>x_0$. By integration, \eqref{uwnega} implies
\begin{align*}
\int_x^\infty (UW)(t)dt\geq (UW)(x)\int_x^\infty
e^{-\int_x^t\frac{1}{\sqrt 2\a(s)}ds}dt. 
\end{align*}
By hypothesis (a), the last integral is infinite, however 
$U,W\in L^2(\RM^+,dt)$ implies that 
$\int_x^\infty (UW)(t)dt<\infty$, contradiction.
\end{proof}

\begin{lemma}\label{pos1} 
$(UW)(t)<0$ for all $t>0$.
\end{lemma} 
\begin{proof} Assume for instance that $U(x_0)=0$. The Cauchy-Lipschitz 
theorem gives $W(x_0)\ne 0$ and the first equation in \eqref{sy1} shows 
that $U(x)$ has the same sign as $W(x)$ for every $x$ in some small interval
$(x_0,x_0+\delta),$ contradicting Lemma \ref{pos}. The same argument works when
$W(x_0)=0$ by considering the second equation in \eqref{sy1}.
\end{proof}

We proved so far that $U$ and $W$ have opposite signs and $\sigma$ is positive.
Condition (c) implies that $\tau$ and $\rho$ have constant signs on $\RM^+$
since $0\leq l\leq m-1$. If $\tau\leq 0$, 
it means that $\tau$ and $\sigma$ have opposite signs, and since also $U$ and $W$ 
have opposite signs by Lemma \ref{pos1}, it follows from the second equation 
in \eqref{sy1} that $W'$ has constant sign. By Lemma \ref{sig} we get 
a contradiction. This shows that $\tau>0$ and similarly we prove $\rho>0$. 
By condition (c), this can only happen for $k=0$, $m=2l+1$ and $(-1)^l=-\e$. 

Assuming this to be the case, the system \eqref{sy1} reads
\be\label{sy2}\begin{split}
 U'=&\tfrac{\b}{4\a^2} U+\tfrac{|\lambda|}{\a} W\\ 
 W'=&\tfrac{|\lambda|}{\a}U+\tfrac{\b}{4\a^2}W.
\end{split}
\ee
The difference $D:=U-W$ is thus a non-vanishing function satisfying 
\be\label{d}D'=\left(\tfrac{\b}{4\a^2}-\tfrac{|\lambda|}{\a}\right)D,\ee
so for every $t_0>0$,
\[D(t)=D(t_0)e^{\int_{t_0}^t\frac{\b(s)}{4\a(s)^2}-
\frac{|\lambda|}{\a(s)}ds}.\]

To conclude the proof of the theorem we distinguish two cases. If $\lambda\le 2^{-3/2}$ we get
$$|D(t)|> |D(t_0)|e^{-\int_{t_0}^t\frac{1}{2\sqrt 2\a(s)}ds}$$
so $D$ cannot be square-integrable because of hypothesis (a).

If $\lambda\ge 2^{-3/2}$, \eqref{ine} together with \eqref{d} show that 
$D$ is decreasing, contradicting Lemma \ref{sig}.
\end{proof}

\section{Axial anomaly for generalized Taub-NUT metrics on $\RM^4$}
\label{sectap}

\subsection{Radial perturbations of the Euclidean metric on $\RM^{2m+2}$}
Recall from Example \ref{ex1} that the Euclidean space
is the metric completion of the CFWP with $\alpha=\frac{t}{\sqrt{2}}$, $\beta=t$
and with basis $B=\CM \P^m$ endowed with the Fubini-Study metric.
Note that by elliptic regularity, bounded spinors 
which are harmonic on a punctured ball $B_0(\epsilon)\setminus \{0\}$ are 
actually smooth and harmonic on $B_0(\epsilon)$, while conversely 
harmonic spinors on $\RM^{2m+2}$ are clearly bounded near $0$.
Theorem \ref{main} applies therefore to the Euclidean metric on 
$\RM^{2m+2}$, for all $m\ge 1$. It is of course well-known that there are 
no harmonic $L^2$ spinors on the Euclidean space. Our results generalize 
this to metrics which are radial perturbations of the standard Euclidean 
metric with any $\alpha,\beta$ satisfying the conditions 
of Theorem \ref{main}. 

\subsection{Generalized Taub-NUT metrics}
The main application of Theorem \ref{main} that we have in mind is the 
vanishing of the index for the generalized Tab-NUT metric of Iwai-Katayama.
The difficulty of the problem resides in the non-Fredholmness of the Dirac 
operator as an 
unbounded operator in $L^2$. Nevertheless in \cite{MV} it was proved 
that the $L^2$ kernel of $D$ is finite-dimensional, and vanishes for the 
standard Taub-NUT metric. 

We cannot apply Theorem \ref{main} directly because of the conformal factor
$\gamma(t)$ in \eqref{tn}. As in Remark \ref{rem4} we set $ds=\gamma(t)dt$. 
Notice that $s=s(t)$, $t=t(s)$ are diffeomorphisms of $\RM^+$ onto itself 
provided 
\begin{align}\label{int}
\int_0^1\gamma(t)dt<\infty,&&\int_0^\infty \gamma(t)dt=\infty.
\end{align}
This condition clearly holds for the conformal factor in \eqref{tn}, which is 
asymptotically constant near infinity and of order $t^{1/2}$ near $t=0$.
Thus we obtain a CFWP metric $\gamma^2(t) g$ where $g$ is itself a CFWP metric. 

\begin{lemma}\label{lemagama}
Let $g$ be a CFWP metric with coefficients $\alpha(t), \beta(t)$, 
and $\gamma(t)$ a conformal factor satisfying \eqref{int}. Then 
the CFWP metric $\gamma^2g$ satisfies the hypotheses of Theorem \ref{main}
if and only if
\begin{enumerate}
\item[(a')] $\displaystyle\int_x^\infty \gamma(t)
e^{-\int_x^t\frac{1}{\sqrt{2}\a(u)}du}dt=\infty$
for all $x>0$;
\item[(b')] $\displaystyle\lim_{t\to 0} \gamma(t)\a(t)=0$;
\item[(c')] $\displaystyle2\a^2(t)\geq \b^2(t)>\tfrac{m-1}{m}2\a^2(t)$ for all $t>0$.
\end{enumerate}
\end{lemma}
\begin{proof}
The coefficients of the CFWP metric $\gamma^2g$ are
\begin{align*}
\tilde{\a}(s)=\gamma(t(s))\alpha(t(s)),&&\tilde{\b}(s)=\gamma(t(s))\beta(t(s))
\end{align*}
so conditions $(b),(c)$ from Theorem \ref{main}
for $\tilde{\a}, \tilde{\b}$ are clearly equivalent 
to conditions (b'), (c'). Now 
$\frac{1}{\a(t)}dt=\frac{1}{\tilde{\a}(s)}ds$ and by definition $\gamma(t)dt=ds$
so by two changes of variables, condition (a') is equivalent to condition
(a) for $\tilde{\a}$.
\end{proof}

As a corollary to Theorem \ref{main} we deduce that the Iwai-Katayama metrics
on $\RM^4$ do not admit non-trivial $L^2$ harmonic spinors. 

\begin{proof}[Proof of Theorem \ref{th1}]
It is straightforward to check that the conditions of Lemma \ref{lemagama} 
hold for the coefficients of Example \ref{ex2}, namely
\begin{align*}
m=1,&&\alpha(t)=\sqrt{2}t,&& \b(t)=\tfrac{2t}{\sqrt{1+ct+dt^2}},&&
\gamma(t)=\sqrt{\tfrac{a+bt}{t}}.
\end{align*}
It follows from
Theorem \ref{main} that there do not exist non-trivial $L^2$ harmonic 
spinors on $(\RM^4\setminus\{0\}, g_{IK})$ bounded near the origin. 
Of course, the metric $g_{IK}$ is smooth at the origin, as can be 
seen by the change of variable $r^2=t$.
In particular we have proved that there do not exist $L^2$ harmonic 
spinors on $(\RM^4,g_{IK})$. 
\end{proof} 

We could have also used the conformal covariance of the Dirac operator 
(cf.\ \cite{hitchin}, see also \cite{nistor}) to related 
harmonic spinors for the metrics
$g$ and $\lambda^2(t) g=g_{IK}$. We do not give details since 
this approach is essentially equivalent to the above proof.

\bibliographystyle{amsplain}

\end{document}